\documentclass[twoside,11pt]{article} 

\usepackage[lofdepth,lotdepth]{subfig}
\usepackage{graphicx}
\usepackage{amsmath}
\usepackage{amssymb}
\usepackage{amsthm}
\usepackage{epsfig}
\usepackage{bbm}
\usepackage[numbers]{natbib}
\usepackage{fullpage, amsfonts, algorithm, algorithmic}
\usepackage{graphicx,color}
\usepackage{lmodern}

\usepackage{float}
\usepackage{multirow}
\usepackage{multicol}
\usepackage{array}
\usepackage{url}
\usepackage{hyperref}
\usepackage{rotating}
\newtheorem{theorem}{Theorem}

\begin{document}






\title{The All-Pairs Vitality-Maximization (VIMAX) Problem}
\author{Alice Paul  \\
alice\_paul@brown.edu, Brown University \\
 \\
  Susan E. Martonosi \\
  martonosi@g.hmc.edu, Harvey Mudd College }

\maketitle

\begin{abstract}
Traditional network interdiction problems focus on removing vertices or edges from a network so as to disconnect or lengthen paths in the network; network diversion problems seek to remove vertices or edges to reroute flow through a designated critical vertex or edge.  We introduce the \textit{all-pairs vitality maximization problem (VIMAX)}, in which vertex deletion attempts to maximize the amount of flow passing through a critical vertex, measured as the all-pairs vitality of the vertex.  The assumption in this problem is that in a network for which the structure is known but the physical locations of vertices may not be known (e.g. a social network), locating a person or asset of interest might require the ability to detect a sufficient amount of flow (e.g., communications or financial transactions) passing through the corresponding vertex in the network.  We formulate VIMAX as a mixed integer program, and show that it is NP-Hard.  We compare the performance of the MIP and a simulated annealing heuristic on both real and simulated data sets and highlight the potential increase in vitality of key vertices that can be attained by subset removal. We also present graph theoretic results that can be used to narrow the set of vertices to consider for removal.  
\end{abstract}




%


\section{Introduction}

Network disruption has important applications to infrastructure design \citep{BCSW2006,MM1970,CSM2004}, energy transmission \citep{CNSW, Hol}, robust network design \citep{CLMR, DWS, Est}, biological systems \citep{RastiVogiatzis_Networks_2022}, illicit trade networks \citep{AnzoomEtAl_IISE_2021}, and counterterrorism \citep{Bas, Sag:2004}.  Much of this work focuses on three primary problem types:  1) network flow interdiction, in which an attacker is trying to decrease the flow capacity of the network by interdicting vertices or edges such that the maximum flow between a source and sink is minimized (e.g., \cite{AEU2010, BaW, BertsimasNasrabadiOrlin, LeiShenSong, LS2007, RoW,W1993, Enayaty-AhangarRainwaterSharkey}); 2) shortest path interdiction, in which an attacker interdicts vertices or edges such that the shortest path between a source and sink is maximized (e.g., \cite{IW2002, PayMerrickSong, ZhangZhuangBehlendorf}); and 3) network diversion, in which a minimum cost, minimal cutset of edges is identified such that when removed, any source-sink path in the network is forced to travel through a particular set of critical edges (e.g., \cite{CCDEGOQ, CullenbineEtAl2013, Cur}).

Of interest in this paper is the concept of vertex (equivalently, edge) \textit{vitality}, which measures the reduction in the maximum flow between the source and sink when that vertex (or edge) is removed from the graph \citep{AusielloEtAl2019, KoschutzkiEtAl2005}.  A vertex having high vitality is needed to achieve a high volume of flow from source to sink, and as such, this vertex will have a high volume of flow passing through it when the maximum flow is achieved.  We define the \textit{all-pairs vitality} of a vertex $v$ to be the summed reduction in the maximum flow between all pairs of nodes (themselves excluding vertex $v$), when vertex $v$ is removed from the graph.

We present the following combinatorial optimization problem, the \textit{all-pairs vitality maximization problem (VIMAX)}:  Given a connected, directed, general capacity graph $G=(V,E)$ with vertex set $V$, edge set $E$, and a key vertex of interest, $k$, identify a subset of vertices $S$, whose removal from the graph $G$ maximizes the all-pairs vitality of $k$. This problem was first introduced in the second author's unpublished manuscript for the specific context of undirected, unit-capacity graphs, for which the maximum flow between a pair of vertices represents the number of edge-disjoint paths between that pair \cite{NetworkDisruption}.

VIMAX can be considered a network disruption problem that is distinct from the three forms outlined above. Covert organizations, such as terrorist groups or drug cartels, tend to communicate along longer paths that are difficult to trace,  suggesting a trade-off between efficiency and secrecy that could render path-length-based attacks ineffective \citep{AnzoomEtAl_IISE_2021, FBW, MGP}.  Moreover, we leverage the possibility that critical vertices in certain types of networks can become vulnerable if they are forced to become more active. (As an example, Osama bin Laden and, subsequently, Ayman al-Zawahiri were known to be leaders of the al-Qaeda terrorist network, yet they remained in hiding for many years before U.S. intelligence could pinpoint their geographic locations.)  If we assume the volume of communication, money, or illicit substances passing through a vertex is a proxy for that corresponding member's visibility to intelligence officers, and communication between pairs of members in the organization is proportional to path capacity, then VIMAX can identify members of the organization whose removal will maximize communication through an important but clandestine leader. Unlike in network diversion problems, we do not require all flow in the remaining graph to be routed through this vertex (indeed in a network diversion problem, the volume of flow passing through the critical vertex might be quite small after vertex or edge removal); instead we seek to \textit{maximize} the total flow routed through this vertex.

In this paper, we examine VIMAX from both computational and theoretical perspectives.  In Section~\ref{sec:lit}, we frame this work in the context of the existing literature.  In Section \ref{sec:optimization}, we define VIMAX, present it as a mixed integer linear program, and demonstrate that it is NP-Hard. Section \ref{sec:SA} presents a simulated annealing heuristic for solving VIMAX.  The computational performance of these two methods is compared in Section \ref{sec:compAnal}. Section \ref{ch:Vertex_Properties} presents mathematical properties of VIMAX that can be leveraged to streamline computations. Section \ref{sec:conc} provides future extensions of this work and concludes.

\section{Literature review}
\label{sec:lit}

We first contrast the network interdiction and diversion problems commonly seen in the literature with the VIMAX problem we will present in this paper. We then discuss the relationship between vitality and other graph centrality metrics.  Finally, we present research on optimization approaches that could be useful to the problem of vitality maximization.  

\subsection{Network Interdiction and Diversion}

Network interdiction models address the logistical problem of removing edges or vertices from a graph to inhibit the flow of resources through a network.  This has applications to military operations and combating drug or human trafficking \cite{KonradTrappPalmbach, TezcanMaass, ZhangZhuangBehlendorf}. Analysis of complex network interdiction typically focuses on disconnecting the network, increasing the lengths of shortest paths, cutting overall flow capacity, or reducing the desirability of paths in the network \citep{AJB, FFV, CavallaroEtAl_2020, GABCH, GCABH, GCLABH, GMR, GMMS, HKYH, HolzmannSmith_2021, MHH, PSS, PayMerrickSong, SLCY, TezcanMaass, WDTZ, ZhangZhuangBehlendorf}.  The most well-known model involves maximum flow network interdiction and its variants \citep{AEU2010, BertsimasNasrabadiOrlin, CMW1998, LeiShenSong, MM1970, P1993, RSL1975, RoW, W1993}.  Of note, \cite{W1993} introduces the ``dualize-and-combine'' method  that is commonly used in network interdiction literature, as well as in this paper. Smith and Song thoroughly survey the network interdiction literature, and demonstrate that the assumptions widely held across the papers they survey make interdiction problems a special case of Stackelberg games \citep{SmithSong2020_EJOR}.

A related problem to network interdiction is the \textit{network diversion problem} in which an attacker seeks to interdict, at minimum cost, a set of edges (equivalently, vertices) such that all source-sink flow must be routed through at least one member of a pre-specified set of ``diversion'' edges or vertices.  This problem was first posed by  \cite{Cur}. 
Applications include military operations, in which it might be beneficial to force a foe to divert its resources through a target edge that is heavily armed; 
and information networks, in which communications are routed through a single edge that can more easily be monitored \citep{LeeChoPark_2019_MOR}.

Cullenbine \textit{et al.} also study the network diversion problem \citep{CullenbineEtAl2013}. They present an NP-completeness proof for directed graphs, a polynomial-time solution algorithm for $s-t$ planar graphs, a mixed integer linear programming formulation that improves upon that given in \cite{Cur}, and valid inequalities to strengthen the formulation.  

Lee \textit{et al.} examine an extension of the network diversion problem known as the \textit{multiple flows network diversion problem} in which there are many source-sink pairs being considered simultaneously \citep{LeeChoPark_2019_MOR}. They define a set $S$ of possible source nodes and $T$ of possible sink nodes.  They are interdicting a \textit{minimum cost} set of edges such that all remaining flow in the network passes through the diversion edge.  They formulate the problem as a mixed integer linear program, and compare its performance to standard combinatorial Benders decomposition and a branch-and-cut combinatorial Benders decomposition.  Without loss of generality, vertex interdiction be formulated as arc interdiction in which each vertex $v$ in the original graph is represented by two vertices $v_i$ and $v_o$ in a modified graph having a single arc between them, $(v_i, v_o)$. Each arc $(u, v)$ in the original graph is then transformed to a corresponding arc $(u_o, v_i)$ in the modified graph. Interdicting this arc in the modified graph is equivalent to interdicting the vertex in the original graph. For undirected graphs, the graph is first transformed into a directed one before doing the transformation. 

There are several aspects of \cite{LeeChoPark_2019_MOR} worth noting as they connect to our work.  First, after the interdiction set is removed from the graph, there is no guarantee that the total flow passing through the diversion edge is particularly large.  In the vitality maximization problem that we present here, we are identifying an interdiction set of vertices such that the flow through the target vertex is maximized, thus ensuring that the target being surveilled has ample flow. Although our formulation does not associate a cost with each vertex that is interdicted, it is disadvantageous for the removal subset to be very large, as that would inherently cause the flow through the target vertex to drop.  Second, we adopt their testing scheme of examining the performance of the algorithms we develop on grid networks (planar), 
as well as random $G_{n,m}$ graphs \cite{knuth2014art}, and a drug trafficking network \cite{natarajan2000}.

A question conversely related to network interdiction and diversion is that of network resilience and detection of attacks. Sharkey \textit{et al.} survey literature on four types of resilience: robustness, rebound, extensibility, and adaptability, with a primary focus on research addressing network robustness and the ability of a network to rebound following an attack \citep{SharkeyEtAl2021}.  Dahan \textit{et al.} study how to strategically locate sensors on a network to detect network attacks \cite{DahanEtAl_2022}.

\subsection{Vitality and Other Graph Centrality Measures}

Vitality is one of several types of graph centrality metrics.  Centrality metrics quantify the importance of a given vertex in a network.  The book of Wasserman and Faust provides a detailed examination of social network analysis stemming from the field of sociology and includes discussion of many commonly known centrality metrics, including degree, betweenness, and closeness \citep{Waf}.  The survey of Rasti and Vogiatzis presents centrality metrics commonly used in computational biology \citep{RastiVogiatzis_AnnalsOR_2019}.  

The degree of a vertex is the number of neighbors it has.  The betweenness of a vertex is the number of shortest paths between all pairs of vertices on which the vertex lies.  Closeness measures the average shortest path length between the vertex and all other vertices in the graph. Vogiatzis \textit{et al.} present mixed integer programming formulations for identifying groups of vertices having the largest degree, betweenness, or closeness centrality in a graph \citep{VogiatzisEtAl_OptLetters_2015}.

Stephenson and Zelen first proposed \textit{information centrality} and applied it to a network of men infected with AIDS in the 1980s \citep{StephensonZelen_1989_SocialNetworks}.  They are among the first to develop a centrality metric that does not require an assumption that information must flow along shortest paths.   They use the theory of statistical estimation to define the information of a signal along the path to be the reciprocal of the variance in the signal.  Assuming the noise induced along successive edges of a path is independent, the variance along each path is additive, and the total variance in the signal grows with the path length.  They then use this assumption to evaluate the total information sent between any pair of vertices $(s,t)$.  From here, they define the centrality of a vertex $i$ to be the harmonic average of the sum of the inverses of the information sent from from vertex $i$ to every other vertex.  They point out that ``information $\ldots$ may be intentionally channeled through many intermediaries in order to `hide' or `shield' information in a way not captured by geodesic paths.''  This appears to be the case in terrorist and other covert networks as well \citep{CKS}. 

Centrality metrics can be used to guide network disruption approaches. Cavallaro \textit{et al.} show that targeting high betweenness vertices efficiently reduces the size of the largest connected component in a graph based on a Sicilian mafia network \cite{CavallaroEtAl_2020}.  Grassi \textit{et al.} find that betweenness and its variants can be used to identify leaders in criminal networks \cite{GrassiEtAl_2019}.

There also exist centrality measures related to network flows, as surveyed in \citep{KoschutzkiEtAl2005}.  In particular, for any real-valued function on a graph, Kosch\"{u}tzki \textit{et al.} define the \textit{vitality} of a vertex (or edge) to be the difference in that function with or without the vertex (or edge). When the function represents the maximum flow between a pair of vertices, the \textit{vitality} of a vertex $k$ in a graph (equivalently, an edge $u$) with respect to an $s-t$ pair of vertices is defined to be the reduction in the maximum flow between $s$ and $t$ when vertex $k$ (equivalently, edge $u$) is removed from the graph. Moreover, when one examines the same reduction in maximum flow in the network over all possible $s-t$ pairs with respect to a given vertex, we have what Freeman \textit{et al.} define as \textit{network flow centrality} \citep{FBW}, or what we refer to as \textit{all-pairs vitality} in this paper.  

The \textit{most-vital edge} or component is the one whose removal decreases the maximum flow through the network by the greatest amount. Identifying the most-vital edge in a network is a long-studied problem dating back to the work of  \cite{CorleyChang1974},  \cite{Wollmer1963}, and \cite{RSL1975}. More recent examination includes the work of \cite{AldersonEtAl2013}, who formulate a mathematical program to maximize resilience, using a defender-attacker-defender model.  They additionally cite several applications for the most-vital edge problem including electric power systems, supply chain networks, telecommunication systems, and transportation.  Ausiello \textit{et al.} provide a method for calculating the vitality of all edges (with respect to a given $s$ and $t$)  with only $2(n-1)$ maximum flow computations, rather than the $m$ computations expected if one were to calculate the vitality of each edge individually \citep{AusielloEtAl2019}.  None of the found literature pertaining to vitality focuses on the problem presented here: that of identifying a set of removal vertices to maximize the vitality of a key vertex (VIMAX).

\section{Optimization framework} \label{sec:optimization}

We will show that VIMAX can be formulated as an integer linear program.  We start by presenting terminology that will be used in the paper.

\subsection{Definitions}

We consider a connected, directed graph $G=(V,E)$ with vertex set $V$, edge set $E$, and a key vertex of interest, $k$.  Each edge $(i,j)$ has a capacity $u_{ij}$ reflecting the maximum amount of flow that can be pushed along that edge.  
The graph has a \textit{key vertex}, $k$, which could represent, for example, an important but elusive participant in an organization. The \textit{vitality maximization problem (VIMAX)} seeks to identify a subset of vertices whose removal from the graph $G$ maximizes the all-pairs vitality of $k$. 
Thus, the objective is to identify a set of vertices to remove from the graph to make the key vertex $k$ as ``active'' as possible by forcing flow to pass through that vertex.  

For any source-sink $s$-$t$ pair, let $z_{st}(G)$ be the value of the maximum $s$-$t$ flow in graph $G$.  We call $Z_k(G)$ the \textit{flow capacity of graph $G$ with respect to vertex $k$}, which is the all-pairs maximum flow in $G$ that does not originate or end at $k$.  Thus,

\begin{equation} \label{eqn:commcap}
Z_k(G) = \sum_{\substack{s,t \in V \setminus \{k \} \\ s \neq t}} z_{st}(G).
\end{equation}

The \textit{all-pairs vitality of $k$}, $\mathcal{L}_k(G)$, equals the flow capacity of the graph with respect to $k$ \textit{minus} the flow capacity with respect to $k$ of the subgraph $G \setminus \{k\}$ obtained when vertex $k$ is deleted:

\begin{equation} \label{eqn:vitality}
\mathcal{L}_k(G) = Z_k(G)-Z_k(G \setminus \{ k \}).
\end{equation}

To measure how the removal of a subset of vertices impacts the vitality of the key vertex, we define the \textit{vitality effect} of subset $S$ on key vertex $k$ to be the change in the key vertex $k$'s vitality caused by removing subset $S$: $\mathcal{L}_k(G\setminus S) - \mathcal{L}_k(G )$. If the vitality effect of $S$ on $k$ is positive, then removing subset $S$ from the graph has diverted more flow through $k$, a desired effect.

The goal of this research is to identify the subset of vertices $S$ 
that maximizes the vitality effect, which is equivalent to maximizing the value of $\mathcal{L}_k(G \setminus S)$.  We formally define the \textit{all-pairs vitality maximization problem} (VIMAX) as \begin{equation} \label{eqn:VIMAX}
max_{ S \subseteq V} \ \mathcal{L}_k(G \setminus S).
\end{equation} 

From expressions (\ref{eqn:commcap}) and (\ref{eqn:vitality}), we see that there is no guarantee that the vitality effect on $k$ of removing any subset $S$ need ever be positive.  When subset $S$ is removed from the graph, the overall flow capacity $Z_k(G\setminus S)$ generally decreases, and never increases, because $S$'s contribution to the flow is removed.  In order for subset $S$'s removal to have a positive vitality effect on key vertex $k$, the remaining flow must be rerouted through $k$ in sufficiently large quantities to overcome the overall decrease in flow through the network.  However, as we will show in Section \ref{subsec:results}, identification of an optimal or near-optimal removal subset often dramatically increases the vitality of the key vertex.

\subsection{Mixed Integer Linear Programming Formulation}
\label{ss:MILP}

To formulate VIMAX as an optimization problem, we first formulate a linear program to solve for the vitality of $k$ in any graph $G$.  Then we expand that formulation into a mixed integer programming formulation that seeks the optimal subset $S$ of vertices to remove from the graph to maximize the vitality of $k$ in the resulting graph.  

\subsubsection{Vitality Max-Flow Subproblems.}

 Following the approach of \cite{IW2002}, 
we take the dual of problem $Z_{k}(G\setminus \{k\})$ to convert it into a minimum cut problem having the same optimal objective function value, and embed it in the formulation of $\mathcal{L}_k(G)$.  Since the dual problem is a minimization problem, the objective function will correctly correspond to the vitality. 
Letting $V' = V \setminus \{k\}$, and letting $E'$ be the set of edges that remain after removing vertex $k$ and its incident edges, we obtain the following linear program for finding $\mathcal{L}_k(G)$:

\begin{equation}
\begin{array}{ll}
\textrm{Maximize} & \displaystyle\sum\limits_{\substack{s,t \in V'\\ s \neq t}} v_{s,t} - \displaystyle\sum\limits_{\substack{s,t \in V' \\ s \neq t}} \displaystyle\sum\limits_{(i,j) \in E'} u_{i,j}y_{i,j,s,t} \\
\textrm{subject to} & \\
& \displaystyle\sum\limits_{j:(i,j) \in E} x_{i,j,s,t} - \displaystyle\sum\limits_{j':(j',i) \in E} x_{j',i,s,t} = \begin{cases} v_{s,t} &\mbox{if } i = s \\ -v_{s,t} &\mbox{if } i = t \\ 0 &\mbox{otherwise} \end{cases}  \forall i \in V , \forall s,t \in V'\\
& \\
& x_{i,j,s,t} \leq u_{i,j} , \forall (i,j) \in E , \forall s,t \in V'\\
& y_{i,s,t} - y_{j,s,t} + y_{i,j,s,t} \geq 0 , \forall (i,j) \in E'  , \forall s,t \in V'\\
& -y_{s,s,t} + y_{t,s,t} \geq 1  , \forall s,t \in V' \\
& \\
& v_{s,t} \geq 0  , \forall s,t \in V'\\
& x_{i,j,s,t} \geq 0  , \forall (i,j) \in E , \forall s,t \in V'\\
& y_{i,j,s,t} \geq 0  , \forall (i,j) \in E'  , \forall s,t \in V' \\
& y_{i,s,t} \textrm{ unrestricted}, \forall i,s,t \in V'. \\
\end{array}
\label{VitalityOptimizationProgram}
\end{equation}

Variables $x_{i,j,s,t}$ and $v_{s,t}$ are the primal variables from the maximum flow formulation of problem $Z_k(G)$.  $x_{i,j,s,t}$ represent the optimal $s-t$ flow pushed along edge $(i,j)$, and $v_{s,t}$ represent the optimal $s-t$ flow values. Variables $y_{i,s,t}$ and $y_{i,j,s,t}$ are the dual variables from the minimum cut formulation of problem $Z_k(G\setminus \{k\})$. We can interpret $y_{i,s,t}$ as vertex potentials: For every edge $(i,j)$, if $y_{i,s,t} < y_{j,s,t}$, meaning vertex $i$ has lower potential than vertex $j$ when computing the minimum $s-t$ cut, then edge $(i,j)$ must cross the cut. In such a case, dual variable $y_{i,j,s,t}=1$, and edge capacity $u_{i,j}$  is counted in the objective function.

\subsubsection{VIMAX: Choosing an Optimal Removal Subset.}

Now that we have expressed the vitality of $k$ in $G$ as a linear program, we can return to VIMAX, which finds a subset $S$ of vertices whose removal maximizes the vitality of $k$.  Given a set $S$, the linear program in 
Equation~\ref{VitalityOptimizationProgram} applied to graph $G \setminus S$ solves for $\mathcal{L}_{k}(G \setminus S)$.  We must modify the LP above to choose a subset $S$ that maximizes the objective function $\mathcal{L}_k(G\setminus S).$

We can formalize this by creating binary variables $z_i$ for each vertex $i$ such that $z_i=1$ if vertex $i$ remains in the graph, and $z_i =0$ if vertex $i$ is removed from the graph (that is, $i$ is included in subset $S$). 
We also define variables $w_{i,j}$ for each edge that indicate whether or not edge $(i,j)$ remains in the graph following the removal of $S$ and/or $k$.  We define linking constraints so that whenever both vertices $i$ and $j$ remain in the graph (that is, $z_i=z_j=1$), then $w_{i,j}$ must equal $1$, and whenever either vertex $i$ or $j$ is selected for deletion (that is, $z_i = 0$ or $z_j =0$ or both) then $w_{i,j}$ must equal $0$.  (Due to this relationship between $w_{i,j}$ and the binary $z_i$, the $w_{i,j}$ are effectively constrained to be binary variables without explicitly declaring them as such.)  

To  Equation \ref{VitalityOptimizationProgram}, we make the following adjustments to the original primal and dual constraints.  We  constrain the primal flow variables $x_{i,j,s,t} \leq u_{i,j}w_{i,j},$ reflecting whether or not edge $(i,j)$ remains in the graph.  We also modify the dual potential constraints so that $y_{i,j,s,t}=0$ whenever vertices $i$ and $j$ are at the same potential (as before) or edge $(i,j)$ no longer exists in the graph.  

Introducing the variables $z_i$ and $w_{i,j}$ and the modifications on our vitality constraints, we can now write the full mixed-integer linear program. Given a graph $G = (V,E)$, a key vertex $k$, and a maximum size, $m$, of the removal set, the following mixed-integer linear program solves VIMAX.

\begin{equation}
\begin{array}{ll}
\textrm{Maximize} & \displaystyle\sum\limits_{\substack{s,t \in V'\\ s \neq t}} v_{s,t} - \displaystyle\sum\limits_{\substack{s,t \in V' \\ s \neq t}} \displaystyle\sum\limits_{(i,j) \in E'}  u_{i,j}y_{i,j,s,t} \\
\textrm{subject to} & \\

& \displaystyle\sum\limits_{\substack{i \in V}} z_i \geq n-m \\
& z_k = 1 \\
& w_{i,j} \leq z_i,  \forall (i,j) \in E \\
& w_{i,j} \leq z_j, \forall (i,j) \in E\\
& w_{i,j} \geq z_i + z_j -1,  \forall (i,j) \in E \\

& \\

& \displaystyle\sum\limits_{j:(i,j) \in E} x_{i,j,s,t} - \displaystyle\sum\limits_{j':(j',i) \in E} x_{j',i,s,t} = \begin{cases} v_{s,t} &\mbox{if } i = s \\ -v_{s,t} &\mbox{if } i = t \\ 0 &\mbox{otherwise} \end{cases} \forall i \in V, \forall s,t \in V'\\
& \\
& x_{i,j,s,t} \leq u_{i,j}w_{i,j} , \forall (i,j) \in E , \forall s,t \in V' \\
& y_{i,s,t} - y_{j,s,t} + y_{i,j,s,t} \geq -(1-w_{i,j}) , \forall (i,j) \in E', \forall s,t \in V'\\
& -y_{s,s,t} + y_{t,s,t} \geq 1  , \forall s,t \in V' \\

& \\

& z_i \textrm{ binary, } \forall i \in V \\
& w_{i,j} \geq 0, \forall (i,j) \in E\\
& v_{s,t} \geq 0  , \forall s,t \in V'\\
& x_{i,j,s,t} \geq 0  , \forall (i,j) \in E , \forall s,t \in V'\\
& y_{i,j,s,t} \geq 0  , \forall (i,j) \in E'  , \forall s,t \in V' \\
& y_{i,s,t} \textrm{ unrestricted, } \forall i,s,t \in V' \\
\end{array}
\label{FinalOptimizationProgram}
\end{equation}

Extending the approach of  \cite{OvaThesis} to general capacity, directed graphs, we can show that VIMAX is NP-Hard.  In the case that $m=1$ and we can remove at most one vertex, we can do brute-force and solve the above MIP setting $z_i = 0$ and all other $z_j = 1$ for all $i \in V'$. 

\begin{theorem} \label{NPHard}
The all-pairs vitality maximization problem is NP-Hard.
\end{theorem}
\begin{proof}

The proof of this can be found in Appendix \ref{sec:AppNPHard}.

\end{proof}

\section{Simulated Annealing Heuristic}
\label{sec:SA}

As an alternative to solving VIMAX exactly with a MIP, we develop a simulated annealing heuristic.  Each iteration of simulated annealing begins with a candidate removal subset.  In the first iteration, this is the empty set, and in subsequent iterations the initial solution is the best solution found at the conclusion of the previous iteration.  The objective function value of each solution is computed as the vitality of the key vertex when this subset is removed from the graph.  Each call to the algorithm consists of an \textit{annealing phase} and a \textit{local search phase}.  

During the annealing phase, neighboring solutions of the current solution are obtained by toggling a single vertex's, or a pair of vertices', inclusion or exclusion from the candidate removal subset, subject to the constraint that $|S| \leq m$.  If the neighboring solution improves the objective function value, it is automatically accepted for consideration.  If the neighboring solution has a worse objective function value, it will be accepted to replace the current solution with an \textit{acceptance probability} governed by a temperature function, $T$. When the temperature is high (in early iterations), there is a high probability of accepting a neighboring solution even if its objective function value is worse than that of the incumbent solution.  This permits wide \textit{exploration} of the solution space.  In later iterations, the temperature function cools, reducing the likelihood that lower objective function value solutions will be considered. This permits \textit{exploitation} of promising regions of the solution space.

Given temperature $T$, the probability of accepting a solution having objective function value $e_0$ when the best objective function value found so far is $e_{max} > e_0$ is given by $P = e^{-(e_{max}-e_0)/T}$.  The initial temperature, $T$, is chosen so that the acceptance probability of a solution having at least 90\% of the initial objective function value is at least 95\%.  In subsequent iterations, $T$ is cooled by a multiplicative factor of $0.95$.

After a set number of annealing iterations, a single iteration of local search is conducted on the best solution found so far by toggling each vertex sequentially to determine if its inclusion or exclusion improves the objective function value.  The best solution found is returned.

We use a Gomory-Hu tree implementation of the all-pairs maximum flow problem to rapidly calculate the vitality of the key vertex on each modified graph encountered by the heuristic \citep{GH1961, G1990}.  For mathematical reasons that are discussed in Section \ref{ch:Vertex_Properties}, we can exclude leaves from consideration in any removal subset.  These two enhancements permit the simulated annealing heuristic to run very fast on even large instances, as we discuss in Section \ref{subsec:results}.

\section{Computational Analysis}
\label{sec:compAnal}

We now present performance comparisons on a variety of datasets of the MIP formulation 
and the simulated annealing heuristic.  Following the approach of \cite{LeeChoPark_2019_MOR}, we generate grid networks, 
which are planar.  We also test the performance of the methods on random networks \cite{knuth2014art} and on a real drug trafficking network \citep{natarajan2000}.  We first describe these data sets and the computational platform used,  and then we present the results. Code and data files are available at our Github repository: \url{https://github.com/alicepaul/network_interdiction}.

\subsection{Data}

\subsubsection{Grid Networks.}

We generate grid networks in a similar fashion as \cite{LeeChoPark_2019_MOR}. We generate square $M \times M$ grids with $M$ varying from five to eight. Such graphs have an edge density of $\frac{4}{M(M+1)}$, which ranges from $13.3\%$ for $M=5$ to $5.6\%$ for $M=8$.
On each grid, we generate edge capacities independently and uniformly at random from the integers from $1$ to $M$.  For each case, we likewise consider two scenarios,  testing a maximum removal subset size of $m=1$ or $m=M$ (that is, $\sqrt{|V|}$).  For each grid size 
and removal subset size combination, we generate three trial graphs.  For each trial, a key vertex is selected uniformly at random over the vertices.

\subsubsection{Random Networks.}
Random $G_{n,m}$ graphs are parametrized by a number of vertices, $n=|V|$, and a number of edges, $m=|E|$ \cite{knuth2014art}. Each graph is sampled by finding a random graph from the set of all connected graphs with $n$ nodes and $m$ edges. We test our methods on graphs having the same number of vertices and same number of edges as the grid networks above: $|V| = \{25, 36, 49, 64\}$ vertices, with $|E| = \{40, 60, 84, 112\}$, respectively. On each graph, we generate edge capacities independently and uniformly at random from the integers from $1$ to $\sqrt{|V|}$.  For each case, we likewise consider two scenarios,  testing a maximum removal subset size of $m=1$ or $m=\sqrt{|V|}$.  For each graph size 
and removal subset size combination, we generate three trial graphs.  For each trial, the key vertex is selected to have the highest betweenness centrality.

\subsubsection{Drug Trafficking Network.}

Lastly, we test our models on a real-world covert cocaine trafficking group, prosecuted in New York City in 1996 \citep{natarajan2000}.  This network consists of 28 people between whom 151 phone conversations were intercepted over wiretap over a period of two months.  An edge exists between persons $i$ and $j$ if at least one conversation between them appears in the data set. There are 40 edges in this graph, corresponding to an edge density of $10.6\%$.  We can consider a unit capacity version of the network, as well as a general capacity version in which the capacity on edge $i-j$ is equal to the number of conversations between them appearing in the data.  The weighted network is shown in Figure \ref{fig:drugnet}, where line width is proportional to the number of wiretapped calls occurring between two operatives.  According to Natarajan \textit{et al.}, some individuals in the network are known to have the roles described in Table \ref{table:drugnetroles}.
  We test a maximum removal subset size of $m=1$ or $m =5 \approx \sqrt{|V|}$.  Because the Colombian bosses (vertices 1, 2, and 3) are high-level leaders important to the functioning of the organization, 
  we treat these vertices as the key vertices on which we attempt to maximize vitality.

\begin{figure}[ht]
\begin{center}
\includegraphics[width = \linewidth]{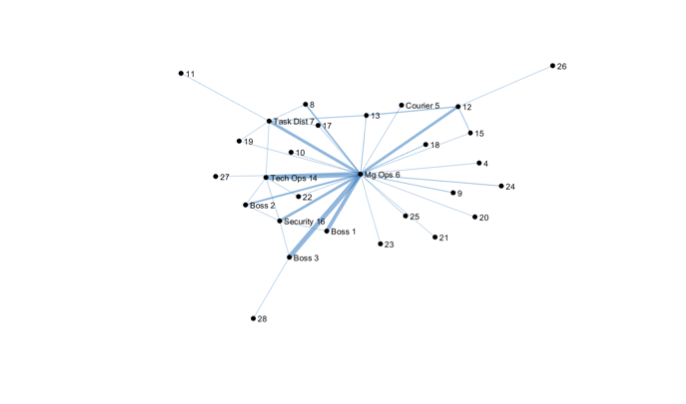}
\caption{Cocaine trafficking network of Natarajan \textit{et al.}  Line width is proportional to number of wiretapped calls made between pairs of operatives \citep{natarajan2000}.}\label{fig:drugnet}
\end{center}
\end{figure}

\begin{table}
\begin{center}
{\small
\begin{tabular}{| c | c |}
\hline
Vertex & Role \\ \hline
1, 2, 3 & Colombian Bosses \\
5 & Courier \\
6 & Managing Operations \\
7 & Task Distribution \\
14 & Technical Operations \\
16 & Security  \\ \hline
\end{tabular} \\
\caption{Roles of notable vertices in the cocaine trafficking network of Natarajan \textit{et al.} \citep{natarajan2000}.}
\label{table:drugnetroles}
}
 \end{center}
\end{table}

\subsection{Computational Framework}


The performance of the MIP and the simulated annealing heuristic was tested on a computer with a 3 GHz 6-Core Intel Core i5 processor and 16 GB of memory. The Single-VIMAX and VIMAX MIP instances were run in python 3.9.6 calling the CPLEX solver through the CPLEX python API, and were each limited to two hours of computation time.  The simulated annealing heuristic was also coded in python and limited to 10,000 iterations on each trial instance. Initial results were collected using the Extreme Science and Engineering Discovery Environment (XSEDE) supercomputers \cite{xsede} and up to five hours of computation time but did not show significantly different results.  In addition to the general VIMAX MIP, a single vertex removal MIP (Single VIMAX) was also tested.  Single vertex removal simulated annealing results are not reported, as they are effectively equivalent to brute force search.

\subsection{Results}
\label{subsec:results}

Table \ref{table:fullVIMAX} presents the results of all completed trials.  The first five columns explain the graph type, number of vertices $(|V|)$, number of edges $(|E|)$ and for the general VIMAX problem allowing multiple removals, the maximum allowed size, $m$, of the removal subset.  Column six gives the initial vitality of the key vertex in the original graph with no vertices removed.  Columns seven through ten provide results on the performance of the single vertex removal MIP (Single VIMAX); columns eleven through fifteen provide results from the multi-removal MIP (VIMAX); and columns sixteen through nineteen provide results from the multi-removal simulated annealing heuristic. (There is no need to use simulated annealing for Single VIMAX because it can be solved by sequentially testing the removal of each vertex.)  For the three methods, the best vitality found within the time or iteration limit, the MIP gap if available, the percentage increase of the best vitality found by the method over the original vitality of the key vertex in the full graph, and the running time in seconds are given.  For the multi-removal methods, the size of the best found removal subset $(|S|)$ is also given.  MIP instances that terminated due to time limit have Time reported as $'-'$.

\begin{sidewaystable}
\begin{center}
{\tiny
\begin{tabular}{|c|c|c|c|c|c||c|c|c|c||c|c|c|c|c||c|c|c|c|}
\hline
\multicolumn{2}{|c|}{}			&		&		&		&		&	\multicolumn{4}{c||}{Single Removal MIP}							&	\multicolumn{5}{|c||}{Multi-Removal MIP}									&	\multicolumn{4}{|c|}{Multi-Removal Simulated Annealing} 							\\ \cline{7-19}
\multicolumn{2}{|c|}{}			&		&		&		&	Orig.	&	Best 	&	MIP	&	\%	&	Time	&	Best 	&	MIP	&	\%	&	Time	&		&	Best Found	&	\%	&	Time	&		\\
\multicolumn{2}{|c|}{Graph Type} 			&	$|V|$	&	$|E|$	&	$m$	&	Vitality	&	Vitality	&	Gap (\%)	&	Incr.	&	(s)	&	Vitality	&	Gap (\%)	&	Incr.	&	(s)	&	$|S|$	&	Vitality	&	Incr.	&	(s)	&	$|S|$	\\ \hline
\multirow{6}*{Drug Network} 	&	\multirow{3}*{unit cap.} 	&	28	&	40	&	5	&	3	&	3	&	0.00	&	0.00	&	15.78	&	8	&	0.00	&	166.67	&	173.38	&	3	&	8	&	166.67	&	182.10	&	3	\\
	&		&	28	&	40	&	5	&	5	&	5	&	0.00	&	0.00	&	11.32	&	8	&	0.00	&	60.00	&	249.92	&	3	&	8	&	60.00	&	185.09	&	3	\\
	&		&	28	&	40	&	5	&	29	&	29	&	0.00	&	0.00	&	9.41	&	31	&	0.00	&	6.90	&	94.11	&	3	&	31	&	6.90	&	210.50	&	3	\\ \cline{2-19}
	&	\multirow{3}*{gen. cap.} 	&	28	&	40	&	5	&	5	&	5	&	0.00	&	0.00	&	13.91	&	5	&	0.00	&	0.00	&	378.17	&	0	&	5	&	0.00	&	196.87	&	3	\\
	&		&	28	&	40	&	5	&	7	&	7	&	0.00	&	0.00	&	11.23	&	7	&	0.00	&	0.00	&	329.90	&	0	&	7	&	0.00	&	225.23	&	1	\\
	&		&	28	&	40	&	5	&	31	&	31	&	0.00	&	0.00	&	9.58	&	31	&	0.00	&	0.00	&	177.60	&	0	&	31	&	0.00	&	264.67	&	0	\\ \hline
\multirow{12}*{Random} 	&	\multirow{3}*{$n = 25$} 	&	25	&	40	&	5	&	0	&	0	&	0.00	&	0.00	&	32.70	&	0	&	0.00	&	0.00	&	1819.28	&	0	&	0	&	0.00	&	244.41	&	0	\\
	&		&	25	&	40	&	5	&	64	&	90	&	0.00	&	40.63	&	25.91	&	135	&	0.00	&	110.94	&	1020.51	&	5	&	115	&	79.69	&	103.65	&	4	\\
	&		&	25	&	40	&	5	&	56	&	73	&	0.00	&	30.36	&	22.04	&	149	&	0.00	&	166.07	&	144.94	&	3	&	149	&	166.07	&	168.73	&	3	\\ \cline{2-19}
	&	 \multirow{3}*{$n=36$} 	&	36	&	60	&	6	&	34	&	34	&	0.00	&	0.00	&	526.01	&	34	&	4548.07	&	0.00	&	-	&	0	&	34	&	0.00	&	543.61	&	0	\\
	&		&	36	&	60	&	6	&	230	&	368	&	0.00	&	60.00	&	533.77	&	693	&	169.68	&	201.30	&	-	&	6	&	859	&	273.48	&	129.36	&	6	\\
	&		&	36	&	60	&	6	&	167	&	304	&	0.00	&	82.04	&	435.19	&	938	&	74.89	&	461.68	&	-	&	5	&	980	&	486.83	&	282.99	&	5	\\ \cline{2-19}
	&	 \multirow{3}*{$n = 49$} 	&	49	&	84	&	7	&	287	&	335	&	0.00	&	16.72	&	6007.47	&	287	&	1900.56	&	0.00	&	-	&	0	&	651	&	126.83	&	229.95	&	7	\\
	&		&	49	&	84	&	7	&	387	&	581	&	0.00	&	50.13	&	4559.45	&	387	&	1632.44	&	0.00	&	-	&	0	&	957	&	147.29	&	452.74	&	6	\\
	&		&	49	&	84	&	7	&	969	&	1,254	&	0.00	&	29.41	&	3519.72	&	969	&	622.02	&	0.00	&	-	&	0	&	2574	&	165.63	&	346.76	&	6	\\ \cline{2-19}
	&	 \multirow{3}*{$n=64$} 	&	64	&	112	&	7	&	151	&	151	&	2377.86	&	0.00	&	-	&	151	&	9144.91	&	0.00	&	-	&	0	&	2225	&	1373.51	&	427.85	&	8	\\
	&		&	64	&	112	&	7	&	581	&	581	&	563.49	&	0.00	&	-	&	581	&	1971.25	&	0.00	&	-	&	0	&	2780	&	378.49	&	897.63	&	8	\\
	&		&	64	&	112	&	7	&	602	&	602	&	492.67	&	0.00	&	-	&	602	&	1887.84	&	0.00	&	-	&	0	&	1619	&	168.94	&	368.37	&	8	\\ \hline
\multirow{12}*{Grid} 	&	\multirow{3}*{$5 \times 5$} 	&	25	&	40	&	5	&	271	&	387	&	0.00	&	42.80	&	75.21	&	559	&	0.00	&	106.27	&	576.68	&	3	&	559	&	106.27	&	186.38	&	3	\\
	&		&	25	&	40	&	5	&	126	&	379	&	0.00	&	200.79	&	54.97	&	472	&	0.00	&	274.60	&	528.42	&	2	&	472	&	274.60	&	195.55	&	2	\\
	&		&	25	&	40	&	5	&	207	&	377	&	0.00	&	82.13	&	58.27	&	432	&	0.00	&	108.70	&	503.27	&	2	&	432	&	108.70	&	202.67	&	2	\\ \cline{2-19}
	&	 \multirow{3}*{$6 \times 6$} 	&	36	&	60	&	6	&	336	&	603	&	0.00	&	79.46	&	1376.30	&	683	&	288.43	&	103.27	&	-	&	2	&	526	&	56.55	&	486.04	&	3	\\
	&		&	36	&	60	&	6	&	79	&	180	&	0.00	&	127.85	&	1219.06	&	293	&	1016.24	&	270.89	&	-	&	6	&	422	&	434.18	&	513.14	&	4	\\
	&		&	36	&	60	&	6	&	222	&	587	&	0.00	&	164.41	&	1071.95	&	1,066	&	195.83	&	380.18	&	-	&	6	&	1178	&	430.63	&	493.66	&	4	\\ \cline{2-19}
	&	 \multirow{3}*{$7 \times 7$} 	&	49	&	84	&	7	&	845	&	1303	&	142.50	&	54.20	&	-	&	845	&	817.34	&	0.00	&	-	&	0	&	2435	&	188.17	&	740.19	&	6	\\
	&		&	49	&	84	&	7	&	651	&	651	&	439.40	&	0.00	&	-	&	651	&	1145.64	&	0.00	&	-	&	0	&	2363	&	262.98	&	743.13	&	5	\\
	&		&	49	&	84	&	7	&	653	&	1446	&	205.77	&	121.44	&	-	&	653	&	1348.54	&	0.00	&	-	&	0	&	3352	&	413.32	&	1224.05	&	3	\\ \cline{2-19}
	&	 \multirow{3}*{$8 \times 8$} 	&	64	&	112	&	7	&	1410	&	2212	&	239.08	&	56.88	&	-	&	1,410	&	1258.79	&	0.00	&	-	&	0	&	6001	&	325.60	&	1592.45	&	7	\\
	&		&	64	&	112	&	7	&	542	&	542	&	1008.47	&	0.00	&	-	&	542	&	2922.76	&	0.00	&	-	&	0	&	1827	&	237.08	&	442.94	&	8	\\
	&		&	64	&	112	&	7	&	380	&	649	&	886.21	&	70.79	&	-	&	380	&	4729.71	&	0.00	&	-	&	0	&	5497	&	1346.58	&	2426.68	&	6	\\ \hline
\end{tabular}} \\
\caption{Computational results of solving VIMAX via mixed integer program (single VIMAX and multi-removal VIMAX) and simulated annealing (multi-removal VIMAX). For the grid and random networks, each row represents a randomly generated instance with randomly selected key vertex.  MIP trials that reached the two-hour time limit show `Time' reported as `-'.  
}
\label{table:fullVIMAX}
\end{center}
\end{sidewaystable}

First we note that Table \ref{table:fullVIMAX} provides a proof-of-concept demonstrating that it is possible to increase (sometimes dramatically) the vitality of the key vertex through subset removal.  Removing a single vertex increased the vitality by 42\%-200\% in all grid network instances for which the MIP solved to optimality within the time limit, and by up to 82\% in the random graph instances; single vertex removal was not able to increase the vitality of the key vertex in the drug network.  When allowing multiple removals, simulated annealing was able to identify removal subsets that increased the vitality on the key vertex by as much as 1,373\%.

Unsurprisingly, the full VIMAX MIP allowing multiple removals is substantially harder to solve than the single removal MIP.  On grid and random networks, the MIP failed to terminate within the two-hour time limit on all instances with at least $n=36$ nodes. On the $n=36$ random and the $7 \times 7$ and $8 \times 8$ grid network instances, the single removal MIP also did not terminate within the time limit, but an improving solution was returned in more cases. The large MIP gaps on the MIP allowing multiple removals indicate a failure to find improving integer solutions. 

For multiple vertex removal, the simulated annealing heuristic yielded excellent solutions in a fraction of the time required  by even the single removal MIP. On the large instances for which the multiple removal MIP reached the time limit, the simulated annealing heuristic found substantially better solutions than the MIP incumbents.  For those instances in which the multiple removal MIP solved to optimality, the solutions found by simulated annealing are often optimal and always near-optimal.

The effectiveness of vertex removal to maximize vitality appears to depend on the network structure and choice of key vertices.  While the drug network has approximately the same number of vertices and edges as the 25-node instances of the random and grid networks, 
the key vertices (corresponding to vertices Boss 1, Boss 2, and Boss 3 in Figure \ref{fig:drugnet}) chosen in these trials are less amenable to vitality maximization.  The drug network has a large number of leaves, whereas the grid networks do not.  As we will see in Section \ref{ch:Vertex_Properties}, vertices, such as leaves, that do not have at least two vertex-disjoint paths to the key vertex will never appear in an optimal removal subset.

Lastly, in these trials, we chose to restrict the removal subset size to at most $m$ vertices.  The reason to restrict the removal subset size is to reduce the solution space, and thus the complexity, of the problem.  This decision is justifiable because we know removing too many vertices will cause overall flow in the network to drop such that the vitality on the key vertex cannot increase.  Thus, an important question is what should be an appropriate value of $m$ to effectively reduce the solution space without compromising the quality of solutions found?  We do not have a definitive answer to this question.  However, we see that in 
many of the trials, the best removal subset identified by any method has a size strictly less than $m \approx \sqrt{|V|}$, suggesting that this choice of $m$ is reasonable for the sizes and types of graphs considered here.

\section{Leveraging Structural Properties of Vitality}
\label{ch:Vertex_Properties}

Thus far, we have established that subset removal can dramatically increase the vitality of a key vertex.  However, solving this problem exactly as a MIP is computationally intractable for even modestly sized graphs.  Fortunately, simulated annealing is an appealing alternative that yields very good solutions in dramatically less time than the MIP.  In this section, we explore mathematical properties that characterize vertices that can be ignored by subset removal optimization approaches.  We demonstrate how these properties can be leveraged to simplify the graph on which VIMAX is run.  

\subsection{Identifying Vitality-\textit{Reducing} Vertices}
\label{ss:vitred}

To reduce the complexity of the optimization formulation, we turn to identifying conditions that cause a vertex to have a vitality-\textit{reducing} effect on the key vertex.  This allows us to ignore such vertices in any candidate removal subset and reduce the solution space of the VIMAX problem.

Our first observation is that the presence of a cycle is necessary for the removal of a vertex to increase the vitality of a key vertex. 
The vitality of a leaf is always equal to $0$, so the removal of any subset that results in $k$ becoming a leaf also cannot increase the vitality of $k$. As a corollary, if $k$ has neighbor set $N(k)$ and more than $|N(k)|-2$ of $k$'s neighbors are removed, the vitality effect on $k$ will be nonpositive.

We can generalize this further.  When there are not at least two vertex-disjoint paths from $i$ to $k$, any removal subset including $i$ will have a vitality effect on $k$ no greater than the same subset excluding $i$, as stated by the following theorem\footnote{A more general cut theorem holds for the specific case of an undirected graph in which all edges in the graph have unit capacity \citep{NetworkDisruption}.  In such a graph, the value of the maximum $s-t$ flow equals the number of edge disjoint paths between $s$ and $t$ in the graph.  In this case, the relationship between the size of the cut between the key vertex $k$ and a candidate for removal, $i$, and the connectivity between vertices along the boundaries of that cut conveys information about the vitality effect on $k$ of removing $i$.  
The reader is also referred to  \cite{paul2012} for an overview of how this theorem might be implemented in practice for unit capacity, undirected graphs.}:

\begin{theorem}
\label{NoCycleThm}  
Let $G$ be a graph with key vertex $k$, and let $i$ be a vertex such that there do not exist at least two vertex-disjoint paths starting at $i$ and ending at $k$. Let $S$ be any vertex subset containing $i$, and let $T = S \setminus \{i\}$. Then, $\mathcal{L}_k(G \setminus S) \leq \mathcal{L}_k(G \setminus T)$. Therefore, $T$ will have at least as large a vitality effect on $k$ as $S$.
\end{theorem}
\begin{proof}

The proof of this can be found in Appendix \ref{sec:AppNoCycleThm}.
\end{proof}

Put simply, the existence of only one vertex-disjoint path between $i$ and $k$ means that $i$ and $k$ do not lie on a cycle.  Therefore when $i$ is removed, any $s-t$ paths that previously passed through $i$ cannot be rerouted through any alternate path passing through $k$.   

Note that identifying vertices that do not have at least two vertex-disjoint paths to $k$ is computationally straightforward.  We can solve an all $u-k$ pairs maximum flow problem on a related graph $\hat{G}$ in which every vertex $u$ is replaced with a pair of vertices connected by a unit capacity edge: $(u,u')$.   For every directed edge $i-j$ in the original graph, we include directed edge $(i',j)$ in the modified graph. Through the use of a Gomory-Hu tree, we can solve this in $O(|V|^3\sqrt{|E|})$ time \citep{GH1961, G1990}.  Any vertex $u$ corresponding to vertex $u^{'}$ in $\hat{G}$ that has a maximum $u^{'}-k$ flow of one in $\hat{G}$ does not have at least two vertex-disjoint paths to $k$ in the original graph and can be ignored by any removal subset.  We call the set of such vertices, $Q$.  Every vertex in $Q$ should be maintained in the graph and not be considered for removal.  

These properties show that when seeking a vitality-maximizing subset for removal, we can ignore all subsets that include:
\begin{itemize}
    \item vertices in $Q$   (i.e. they do not share a cycle with $k$);      \item more than $|N(k)|-2$ of $k$'s neighbors.
    \end{itemize}

After performing preprocessing on the graph to identify $N(k)$ and $Q$, we can add the following constraints to the MIP formulation:  

\begin{equation}
\begin{array}{ll}
& z_i = 1, \forall i \in Q \\
& \displaystyle\sum\limits_{\substack{i \in N(k)}} z_i \geq 2 \\
& \\
\end{array}
\label{FinalOptimizationProgram_reduced}
\end{equation}

Although the above constraints provide a tighter formulation for VIMAX, the anticipated benefits of these constraints are likely to be modest.  Table \ref{table:simplify} shows $|Q|$ (the number of vertices that do not have at least two vertex-disjoint paths to $k$) for each graph used for testing in Section \ref{sec:compAnal}.

\begin{table}
\begin{center}
{\small
\begin{tabular}{|c|c|c|c||c||c|c| c | c|}
\hline
\multicolumn{2}{|c|}{}			&		&		&		&		&		&		&		\\
\multicolumn{2}{|c|}{}			&		&		&		&		&		&	\% Decr	&	\% Inc	\\
\multicolumn{2}{|c|}{Graph Type} 			&	$|V|$	&	$|E|$	&	$|Q|$	&	$|\hat{V}|$	&	$|\hat{E}|$	&	Time	&	Obj	\\ \hline
\multirow{6}*{Drug Network} 	&	\multirow{3}*{unit cap.} 	&	28	&	40	&	14	&	18	&	30	&	93.19	&	0.00	\\
	&		&	28	&	40	&	14	&	18	&	30	&	94.68	&	0.00	\\
	&		&	28	&	40	&	13	&	18	&	30	&	91.83	&	0.00	\\ \cline{2-9}
	&	\multirow{3}*{gen. cap.} 	&	28	&	40	&	14	&	20	&	32	&	89.40	&	0.00	\\
	&		&	28	&	40	&	14	&	20	&	32	&	92.91	&	0.00	\\
	&		&	28	&	40	&	13	&	20	&	32	&	90.98	&	0.00	\\ \hline
\multirow{12}*{Random} 	&	\multirow{3}*{$n = 25$} 	&	25	&	40	&	23	&	12	&	11	&	99.98	&	0.00	\\
	&		&	25	&	40	&	1	&	25	&	40	&	62.92	&	0.00	\\
	&		&	25	&	40	&	6	&	24	&	38	&	23.98	&	0.00	\\ \cline{2-9}
	&	 \multirow{3}*{$n=36$} 	&	36	&	60	&	33	&	5	&	4	&	99.99	&	0.00	\\
	&		&	36	&	60	&	4	&	35	&	59	&	-	&	13.56	\\
	&		&	36	&	60	&	8	&	35	&	59	&	-	&	0.00	\\ \cline{2-9}
	&	 \multirow{3}*{$n = 49$} 	&	49	&	84	&	6	&	48	&	83	&	-	&	0.00	\\
	&		&	49	&	84	&	7	&	47	&	82	&	-	&	0.00	\\
	&		&	49	&	84	&	7	&	48	&	83	&	-	&	0.00	\\ \cline{2-9}
	&	 \multirow{3}*{$n=64$} 	&	64	&	112	&	8	&	64	&	112	&	-	&	0.00	\\
	&		&	64	&	112	&	10	&	63	&	111	&	-	&	0.00	\\
	&		&	64	&	112	&	10	&	63	&	111	&	-	&	0.00	\\ \hline

\end{tabular}} \\
\caption{Improvement in key VIMAX instance size parameters by identifying vitality-reducing vertices and using graph-simplification.  $|Q|$ is the number of vertices that do not have at least two vertex-disjoint paths to $k$; vertices in $Q$ can be ignored by VIMAX (see Section \ref{ss:vitred}).  $|\hat{V}|$ and $|\hat{E}|$ are the numbers of vertices and edges, respectively, in the reduced graph after applying the graph simplification method of Section \ref{ss:simplification}. The last two columns report the percentage decrease in time and percentage increase in best objective function value of the graph simplification method compared to the Multi-Removal MIP results reported in Table \ref{table:fullVIMAX}. Entries denoted by '-' indicate instances in which the MIP did not terminate within two hours.}
\label{table:simplify}
 \end{center}
 \end{table}

Unsurprisingly given their structure, all the vertices in the grid networks have at least two vertex-disjoint path to $k$; thus none of these vertices can be eliminated from consideration and are omitted from Table~\ref{table:simplify}.  By contrast, the sparse drug trafficking network has nearly half of its vertices that do not have at least two vertex-disjoint paths to the key vertex; this is a significant reduction in the number of candidate vertices for removal, but VIMAX was readily tractable on this already-small network.  
Thus, this criterion alone is unlikely to render previously intractable MIP instances tractable.

\subsection{Simplifying the Graph} \label{ss:simplification} 

Because VIMAX grows rapidly in the number of vertices, we can improve the computational tractability of VIMAX by simplifying our original graph into a vitality-preserving graph having fewer vertices. We rely heavily on Theorem~\ref{NoCycleThm} to do this.

Suppose that a vertex $v$ disconnects the graph into two components $T_1$ and $T_2$ such that $k\in T_1$. Then, by Theorem~\ref{NoCycleThm}, an optimal solution will not contain any vertex in $T_2$. Further, the maximum flows between pairs of vertices within $T_2$ do not contribute to the vitality effect on $k$. Therefore, all that is needed to preserve the vitality effect on $k$ in the simplified graph is to preserve information about the maximum flow between all pairs of vertices $s,t$ such that $s \in T_1$ and $t \in T_2$. 

For all vertices $t \in T_2$ we create a single edge between $t$ and $v$ with capacity equal to the maximum flow between $t$ and $v$. This replaces all previous edges between vertices in $T_2$. This affects the value of the all-pairs maximum flow problem but does not affect the vitality effect on $k$ for any subset $S \subset T_1$. Further, if any subset of vertices $T' \subseteq T_2$ all have the same new capacity value, we combine $T'$ into a single vertex with weight $|T'|$. When calculating the maximum flow between any pair of vertices $s$ and $t$ in the graph, we multiply the flow by the product of the weights of the vertices to account for this simplification.

Using the process described in the previous section, we can identify the subset of vertices $Q \subseteq V \setminus \{k\}$ that do not have at least two vertex-disjoint paths to $k$. Given a vertex $i \in Q$, we find a path from $i$ to $k$ and find the first vertex $v$ along $i$'s path to $k$ such that $v$ has at least two vertex-disjoint paths to $k$.  Removing the vertex $v$ disconnects the graph. Therefore, we follow the simplification process above and mark all vertices in the corresponding $T_2$, including $i$, as processed. We then repeatedly identify any unprocessed vertex in $Q$ to further simplify the graph. After all vertices in $Q$ have been processed, all these vertices will be weighted leaves in the new simplified graph where the weight depends on how many vertices have been combined. All other vertices will retain a weight of one.

Figure~\ref{fig:SimplifiedGraph} shows an example of this simplification process in which there are two components that have been simplified. Note that vertices 4, 6, and 7 have been combined together into a vertex with weight three. Further, vertices 5 and 8 have been combined together into a vertex with weight two.

\begin{figure}
\begin{center}
\includegraphics[width=12cm]{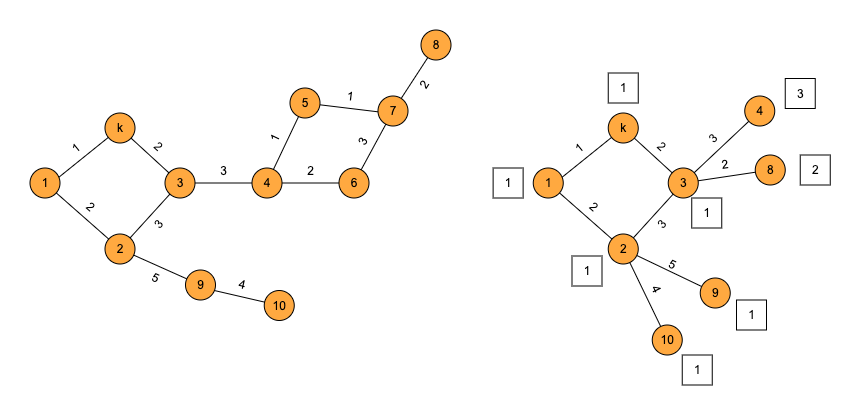}
\caption[An example of how a graph can be simplified by condensing components that do not have at least two vertex-disjoint paths to $k$]{\label{fig:SimplifiedGraph} An example of a graph (left) and its simplified version (right) with vertex weights. Vertices 4, 6, and 7 have been combined together into a vertex with weight three. Further, vertices 5 and 8 have been combined together into a vertex with weight two.}
\end{center}
\end{figure}

As argued above, the maximum flow between all pairs of vertices that were in the same simplified component never contribute to the vitality effect on $k$. Therefore, we ignore these pairs in the optimization problem by removing the appropriate variables and constraints. We therefore just need to check that we have preserved the maximum flow between all pairs of vertices that were not in the same component. This is true by nature of the weights which are multiplied. For example, in Figure~\ref{fig:SimplifiedGraph}, we multiply by weight 4 for the maximum flow between vertex 4 and vertex 1, accounting for all the paths between vertices 4, 6, and 7 and vertex 1. Thus, our optimization problem still finds an optimal subset to remove on the simplified graph that is optimal in the original graph. The number of pairs of vertices decreases from 45 to 19 since the number of vertices excluding $k$ decreases from 10 to 7 and we can ignore the flow between vertices 9 and 10 and between vertices 4 and 8 in the simplified graph.

Table \ref{table:simplify} shows the number of vertices ($|\hat{V}|$) and edges ($|\hat{E}|)$ in each test graph after applying the graph simplification algorithm.  The only graph types experiencing an appreciable reduction in size after simplification are the drug trafficking network and the smaller random graphs.  We posit that highly connected graphs such as the grid networks 
are less amenable to the simplification method than sparser networks. In Table \ref{table:simplify} we also include the percentage decrease in time and percentage increase in the best objective function value found via graph simplification to the Multi-Removal MIP removal results reported in Table~\ref{table:fullVIMAX}. The time includes the time to perform the graph simplification, which is very efficient. For graphs with a significant reduction in the number of nodes and edges, we see a corresponding decrease in the runtime for the MIP. For the larger networks that did not terminate within the time limit, we only see the best vitality found improve in one instance.

\section{Future Work and Conclusions} \label{sec:conc}

In this paper we have presented the VIMAX optimization problem that identifies a subset of vertices whose removal maximizes the volume of flow passing through a key vertex in the network.  VIMAX is NP-Hard. We have used the dualize-and-combine method of \cite{W1993} to formulate VIMAX as a mixed integer linear program, and we compared its performance to that of a simulated annealing heuristic.  We also demonstrated how identifying vertices not having at least two vertex-disjoint paths to the key vertex can be used to simplify the graph and reduce computation time on certain graph types.

Additionally, this paper opens up a rich area of future research.

\begin{itemize}
\item Computational improvements - Graph Simplification: Additional properties of vitality-reducing vertices, such as those outlined in \cite{paul2012} for the unit capacity case, could be derived for the general capacity case and used to preprocess or simplify the graph to reduce the solution space of VIMAX.  In particular, it would be beneficial to identify small cuts in the graph such that all vertices on the other side of the cut as $k$ can be ignored from consideration.

\item Computational improvements - Bender's Decomposition: Because the number of constraints in the VIMAX MIP grows on the order of $O(|E||V|^2)$, we can use Bender's decomposition algorithm to solve our problem for large graphs. The decomposition is presented in Appendix~\ref{sec:benders}, but preliminary testing did not improve the MIP performance. The survey of Smith and Song illustrates a variety of approaches that could be applied to improve the performance of the Bender's decomposition of VIMAX \cite{SmithSong2020_EJOR}.

\item Optimization: In this paper, we have focused on identifying vertices having high vitality effect on the key vertex without considering the cost or difficulty of removing them from the graph.  An enhancement to VIMAX could include a budget constraint restricting the choice of subsets based on the difficulty of their removal.  
\item Game theory and dynamic response: The disruption technique described in this paper focuses on the network at one snapshot in time and assumes that any subset removal occurs simultaneously and that the network remains static. Extensions to VIMAX might explore cascading effects of sequential vertex removal, similar to the literature on multi-period interdiction \citep{Enayaty-AhangarRainwaterSharkey}, cascading failures \citep{CLM, MoL,ZPLY}, agent-based models for counter-interdiction responses \cite{MaglioccaEtAl_2019}, and game theoretic responses of the network to disruptions, such as adding new edges.  
\item Imperfect information: The VIMAX formulation presented here assumes complete and perfect knowledge of the network's structure.  However, the complete structure of a covert network is typically not known to enforcement agencies, and can evolve rapidly \cite{KonradTrappPalmbach}.  Future work could address applying VIMAX to networks with uncertain or unknown structure. 
\item Robust network design: We can use the results of this research to design networks, such as telecommunication and other infrastructure networks, to be robust to vitality-diverting attacks \citep{CLMR}.
\item Multiple key vertices: In the case that we want to maximize the flow through a subset $S$ of key vertices, we can extend the definition of vitality maximization to maximize the all-pairs vitality of $S$. The MIP and simulated annealing algorithm can be updated accordingly.
\end{itemize}

VIMAX has broad applicability to problems including disrupting organized crime rings, such as those used in terrorism, drug smuggling and human trafficking; disrupting telecommunications networks and power networks; as well as robust network design. 

\section{Acknowledgements}
This work used the Extreme Science and Engineering Discovery Environment (XSEDE) \citep{xsede}, which is supported by National Science Foundation grant number ACI-1548562.  Specifically, this work used the XSEDE Bridges-2 Extreme Memory and Regular Memory supercomputers at the Pittsburgh Supercomputing Center through allocation MTH210021.  We thank consultant T. J. Olesky for their assistance troubleshooting batch calls to AMPL, which was made possible through the XSEDE Extended Collaborative Support Service (ECSS) program \citep{ecss}.  The authors would also like to acknowledge Doug Altner, Michael Ernst, Elizabeth Ferme, Sam Gutekunst, 
Danika Lindsay, Yaniv Ovadia,  
Sean Plott, and Andrew S. Ronan for their contributions to early efforts in this work \citep{NetworkDisruption, GutekunstThesis, OvaThesis}.  This work was supported by  the National Science Foundation Research Experiences for Undergraduates program (NSF-DMS-0755540).

%
 \appendix
 
 \section{Proof of Theorem \ref{NPHard}}
\label{sec:AppNPHard}
In this section we prove Theorem \ref{NPHard} stating that the all-pairs vitality maximization problem is NP-Hard.  Our proof extends the proof of \cite{OvaThesis} for the special case of undirected, unit-capacity edges. We first restate VIMAX as a decision problem: For a fixed value $C$, does there exist a subset $S$ such that $\mathcal{L}_k(G\setminus S) \geq C$? 

\begin{theorem}
The all-pairs vitality maximization problem is NP-Hard.
\end{theorem}
\begin{proof}

We use a reduction from the 3-Satisfiability problem (3SAT). Given an instance of 3SAT with $n$ boolean variables $x_1, x_2, \ldots, x_n$ and $m$ clauses in 3-conjunctive normal form $c_1, c_2, \ldots, c_m$, the 3SAT decision problem is whether there is an assignment of variables to true/false values such that all clauses are satisfied. As an example with three variables, any assignment with $x_3$ set to false would satisfy the two clauses $(x_1 \text{ or } \overline{x_2} \text{ or } \overline{x_3})$ and $(\overline{x_1} \text{ or } x_2 \text{ or } \overline{x_3})$.

Given an instance of 3SAT, we construct a corresponding instance of VIMAX. We start building our directed graph $G$ with three vertices $d_1$, $k$ (the key vertex), and $d_2$ with an edge from $k$ to $d_2$ with capacity $n+m$. Further, for each variable $x_i$ we create four vertices $\{a_i, b_i, t_i, f_i\}$ and add edges $(d_1, a_i)$, $(a_i, t_i)$, and $(a_i, f_i)$ each with capacity two and edges $(t_i, b_i)$, $(f_i, b_i)$, $(t_i, d_2)$, $(f_i, d_2)$, and $(b_i, k)$ each with capacity one.

Then, for each clause $c_j$, we create two variables $u_j$ and $v_j$ and add unit capacity edges $(d_1, u_j)$ and $(v_j, k)$. To encode this clause, for each variable $x_i$ in clause $c_j$ we add unit edges $(u_j, t_i)$ and $(t_i, v_j)$;  for each variable $\overline{x_i}$ in clause $c_j$ we add unit edges $(u_j, f_i)$ and $(f_i, v_j)$. Last, we create $M = 8 \cdot (m+n+n \cdot m)$ leaves with unit edges to $d_1$ and $M$ leaves with unit edges from $d_2$ and set $C=(M+1)^2(n+m)$. An example graph of a single-clause, three-variable, 3SAT problem having clause ($\overline{x_1}$ or $x_2$ or $\overline{x_3}$) is given in Figure~\ref{fig:np_graph}. 

\begin{figure}
\begin{center}
\includegraphics[width=12cm]{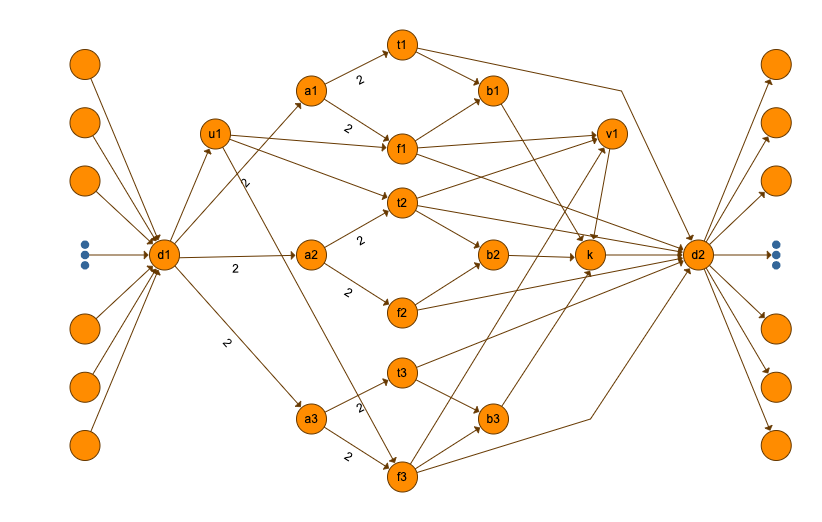}
\caption[Graph representation of a single clause 3SAT problem with clause ($\overline{x_1}$ or $x_2$ or $\overline{x_3}$).]{\label{fig:np_graph} Graph representation of a single clause 3SAT problem with three variables and the clause ($\overline{x_1}$ or $x_2$ or $\overline{x_3}$).  All edge capacities equal one except where indicated otherwise.}
\end{center}
\end{figure}

Note that the leaves adjacent to $d_1$ and $d_2$ essentially increase the weight of the flow between $d_1$ and $d_2$. In particular, if we define 
$$\mathcal{L}^{s,t}_k(G \setminus S) := z_{st}(G \setminus S) - z_{st}(G\setminus (S \cup \{k\}))$$
and let $V'$ be all vertices excluding these leaves as well as $d_1$, $d_2$, and $k$, then we can rewrite the all-pairs vitality as
\begin{align*}
& \mathcal{L}_k(G \setminus S) \\
& = \sum_{\substack{s,t \in V \setminus (S \cup \{k \}) \\ s \neq t}}  \mathcal{L}^{s,t}_k(G \setminus S) \\
& = (M+1)^2 \mathcal{L}^{d_1,d_2}_k(G \setminus S) + (M+1) \sum_{\substack{s \in V' \setminus S }}  \left[\mathcal{L}^{d_1,s}_k(G \setminus S) +  \mathcal{L}^{s,d_2}(G \setminus S) \right] + \sum_{\substack{s,t \in V' \setminus S \\ s \neq t}}  \mathcal{L}^{s,t}_k(G \setminus S) \\
& = (M+1)^2 \mathcal{L}^{d_1,d_2}_k(G \setminus S) + (M+1) \sum_{\substack{s \in V' \setminus S }}   \mathcal{L}^{s,d_2}(G \setminus S)  .
\end{align*}
The last line holds since paths from $d_1$ to $s \in V'\setminus S$ or between $s$ and $t \in V' \setminus S$ cannot travel through $k$. Further, we can bound the second half of the sum above by bounding the vitality by the capacity out of the starting node for each maximum flow. 
\begin{align*}
(M+1) \sum_{\substack{s \in V' \setminus S }}   \mathcal{L}^{s,d_2}(G \setminus S)  & \leq (M+1)(4m+9n+2m\cdot n) \\
& \leq  \frac{1}{2} (M+1)^2.
\end{align*}
This shows that the maximum flow from pairs that are not $\{d_1, d_2\}$ contributes a trivial amount to the overall vitality. Therefore, finding a subset such that $\mathcal{L}_k(G \setminus S) \geq C = (M+1)^2(n+m)$ is equivalent to finding a subset $S$ such that $\mathcal{L}^{d_1,d_2}_k(G \setminus S) \geq n+m$. 

We now show that given an assignment of variables to boolean values that satisfy all clauses, we can find an equivalent subset $S$ such that $\mathcal{L}^{d_1, d_2}_k(G \setminus S) \geq n+m$. Let $S$ contain $t_i$ for all $i$ such that $x_i$ is set to false and $f_i$ for all $i$ such that $x_i$ is set to true. 

Consider the maximum flow between $d_1$ and $d_2$ in $G \setminus S$. For each variable $x_i$ such that $t_i \in S$, we send two units of flow: one along the path ($d_1$--$a_i$--$f_i$--$b_i$--$k$--$d_2$) and one along the path ($d_1$--$a_i$--$f_i$--$d_2$). If, instead, $f_i \in S$, then the paths change to use $t_i$ instead of $f_i$. Further for each clause $j$, since this clause is satisfied, there exists at least one vertex $t_i$ or $f_i$ adjacent to $u_j$ that is not in $S$. Without loss of generality, let this vertex be $t_i$. We send one unit of flow along the path ($d_1$--$u_j$--$t_i$--$v_j$--$k$--$d_2$). The overall flow has value $2n+m$. Since all edges adjacent to $d_1$ are saturated, this is a maximum flow. 

Now consider the maximum flow between $d_1$ and $d_2$ in $G \setminus (S \cup \{k\})$. For each variable $x_i$ such that $t_i \in S$, we send one unit of flow along the path ($d_1$--$a_i$--$f_i$--$d_2$). If, instead, $f_i \in S$, then the path changes to use $t_i$ instead of $f_i$. The overall flow has value $n$. Since all edges adjacent to $d_2$ are saturated in  $G \setminus (S \cup \{k\})$ this is a maximum flow. This shows that $\mathcal{L}^{d_1, d_2}_k(G \setminus S) \geq n+m$.

We must now show the reverse direction to complete the proof. Suppose that we have found a subset $S$ such that $\mathcal{L}_k(G\setminus S) \geq C$. Then, given that all pairs except $d_1$ and $d_2$ contribute at most $\frac{1}{2}(M+1)^2$ to the vitality, it must be the case that $\mathcal{L}^{d_1,d_2}_k(G \setminus S) \geq n+m$. We  decompose the flow into unit flow paths from $d_1$ to $d_2$. Let $f(s,t)$ be the number of these paths that go from $s$ to $t$ in the maximum flow from $d_1$ to $d_2$ in $G \setminus S$ and $f'(s,t)$ be the number of paths from $s$ to $t$ in the maximum flow between $d_1$ and $d_2$ in $G \setminus (S \cup \{k\})$. Then,
\begin{equation}
    \mathcal{L}^{d_1,d_2}_k(G\setminus S) = \sum_{i=1}^n [f(a_i,d_2)-f'(a_i,d_2)] + \sum_{j=1}^m [f(u_j,d_2)-f'(u_j,d_2)] .
\label{eqn:np_flow}
\end{equation}
For the first term in Equation~\ref{eqn:np_flow}, we can verify that $[f(a_i,d_2)-f'(a_i,d_2)]\leq 1$ if exactly one of $t_i$ and $f_i$ is in $S$ and $\{a_i, b_i\} \cap S = \emptyset$ and at most zero otherwise. In particular, if $t_i$ and $f_i$ are both  in $S$ then $f(a_i,d_2)=f'(a_i, d_2) = 0$. If both $t_i$ and $f_i$ are not in $S$, then at most two units of flow can go from $a_i$ to $d_2$ in both graphs and both $t_i$ and $f_i$ can avoid using vertex $k$. Only when exactly one of $t_i$ or $f_i$ has been chosen will at least one path be forced to go through vertex $k$. For the second term, each term is also at most one given the unit capacity of the edge from $d_1$ into $u_j$. Therefore, 
\begin{equation}
   \mathcal{L}^{d_1,d_2}_k(G\setminus S) =\sum_{i=1}^n [f(a_i,d_2)-f'(a_i,d_2)] + \sum_{j=1}^m [f(c_j,d_2)-f'(c_j,d_2)] \leq n+m.
\label{eqn:np_flow2}
\end{equation}
Since $\mathcal{L}^{d_1,d_2}_k(G\setminus S) \geq n+m$ this implies equality throughout and that $|\{t_i, f_i \} \cap S| = 1$ for all $i = 1, 2, \ldots, n$. For each variable for which $t_i$ is in $S$, we set that variable to false. Otherwise, we set the variable to true. Last, in order for every clause to contribute at least one to the overall vitality, $u_j$ must be adjacent to some $t_i$ or $f_i$ not in $S$. Given the design of our network, this indicates that the assignment satisfies that clause. 

Overall, this shows that every 3SAT decision problem can be reduced to a VIMAX decision problem and that VIMAX is NP-Hard. 
\end{proof}

\section{Proof of Theorem \ref{NoCycleThm}}
\label{sec:AppNoCycleThm}

Here we prove Theorem \ref{NoCycleThm} stating that the removal of any vertex not having at least two vertex-disjoint paths to the key vertex $k$ can never increase the vitality of $k$.

\begin{theorem}
Let $G$ be a graph with key vertex $k$, and let $i$ be a vertex such that there do not exist at least two vertex-disjoint paths starting at $i$ and ending at $k$. Let $S$ be any vertex subset containing $i$, and let $T = S \setminus \{i\}$. Then, $\mathcal{L}_k(G \setminus S) \leq \mathcal{L}_k(G \setminus T)$. Therefore, $T$ will have at least as large a vitality effect on $k$ as $S$.
\end{theorem}
\begin{proof}

Let $G$ be a graph with key vertex $k$ and let $i$ be a vertex such that there do not exist at least two vertex-disjoint paths starting at $i$ and ending at $k$.  Then there exists a cut vertex $v$ whose removal would disconnect the graph into at least two components.  We consider two cases, $v \neq i$ and $v = i$.

When $v \neq i$, then $v$ separates a component $G_k$ that includes $k$ from a component $G_i$ that includes $i$.  Consider the maximum flow between an $s-t$ pair ($s,t \neq k$).  
\begin{itemize}
\item If both $s$ and $t$ are in $G_i$, the flow between them is unaffected by the removal of vertex $k$, whether or not vertex $i$ is removed from the graph. This is because any optimal flow path that passes through vertex $k$ must first go into and out of vertex $v$, creating a flow cycle,  $s -  \ldots - v -\ldots - k -\ldots - v - \ldots - t$, and thus is equivalent to a flow path that avoids $G_k$ entirely, $s -  \ldots - v - \ldots - t$.
\item If both $s$ and $t$ are in $G_k$, their contribution to the vitality of $k$ is unaffected by the removal of $i$ by the same logic as above:  any optimal flow path that passes through vertex $i$ must go into and out of vertex $v$, creating a flow cycle, and thus is equivalent to a flow path that avoids $G_i$ entirely.
\item If, without loss of generality, $s \in G_i$ and $t \in G_k$, then the removal of vertex $i$ may reduce the flow between $s - \ldots - v$, but the remainder of the path $v - \ldots - t$ is unaffected.  Thus no additional flow can be routed through $k$ when $i$ is removed than when $i$ is present.
\end{itemize}

When $v = i$, then $i$ separates a component $G_k$ that includes $k$ from the remainder of the graph, $G_i$.  In this case, the removal of $i$ will eliminate all $s-t$ flow between $s \in G_i$ and $t \in G_k$, regardless of whether or not $k$ is in the graph.  Thus, no additional flow can be routed through $k$ when $i$ is removed from the graph than when $i$ is present.\end{proof}

\section{Benders Decomposition}
\label{sec:benders}

Because the number of constraints in the VIMAX MIP grows on the order of $O(|E||V|^2)$, we can use Benders decomposition algorithm to solve our problem for large graphs. 
In our case, the integer master problem chooses the subset of vertices to remove; this problem has relatively few variables and  constraints. Given a fixed removal subset, we are left with a large linear network flow subproblem that is guaranteed to have an integer optimal solution. 

We see in Equation~\ref{FinalOptimizationProgram}  constraints that couple $w_{i,j}$, $x_{i,j,s,t}$, $y_{i,j,s,t}$ and $y_{i,s,t}$. We let the $z_i$'s and $w_{i,j}$'s be the variables in our master problem. Our initial master problem contains only the constraints related to the $w_{i,j}$'s and $z_{i}$'s, representing the choice of subset to remove. Thus the master problem is

\begin{equation}
\begin{array}{ll}
\textrm{Maximize} &  \mathcal{L}_k \\
\textrm{subject to} & \\ 
& \sum_{i \in V} z_i \geq n-m \\
& z_k = 1 \\
& w_{i,j} \leq z_i, \forall (i,j) \in E \\
& w_{i,j} \leq z_j, \forall (i,j) \in E \\
& w_{i,j} \geq z_i + z_j -1, \forall (i,j) \in E\\

& \\

& z_i \textrm{ binary, } \forall i \in V \\
& w_{i,j} \geq 0, \forall (i,j) \in E\\
& \mathcal{L}_k \geq 0. \\
\end{array}
\label{MLIPMaster}
\end{equation}

Here, $\mathcal{L}_k$ represents the optimal vitality of $k$. It currently has no restrictions on its value. 

Solving Equation~\ref{MLIPMaster} determines a feasible $\mathbf{z}$ and $\mathbf{w}$, which we can use to compute the vitality of $k$ in the dual of the linear subproblem. When taking the dual we let $x^{'}_{i,s,t}$ be the dual variables corresponding to the flow balance constraints of the $x_{i,j,s,t}$'s and $x^{'}_{i,j,s,t}$ be the dual variables corresponding to the capacity constraints on the $x_{i,j,s,t}$'s. Similarly, we let $y^{'}_{i,j,s,t}$ be the dual variables corresponding to the edge constraints on $y_{i,j,s,t}$, and we let $y^{'}_{s,t}$ be the dual variables corresponding to the constraints on the relationship between $y_{s,s,t}$ and $y_{t,s,t}$. The linear subproblem becomes

\begin{equation}
\begin{array}{ll}
\textrm{Minimize} & \displaystyle\sum\limits_{\substack{s,t \in V' \\ s \neq t}} \displaystyle\sum\limits_{(i,j) \in E}  u_{i,j}w_{i,j} x^{'}_{i,j,s,t}+\displaystyle\sum\limits_{\substack{s,t \in V'\\ s \neq t}} y^{'}_{s,t} -  \displaystyle\sum\limits_{\substack{s,t \in V' \\ s \neq t}} \displaystyle\sum\limits_{(i,j) \in E'}(1-w_{i,j}) y^{'}_{i,j,s,t}\\
\textrm{subject to} & \\ 

& x^{'}_{i,s,t} - x^{'}_{j,s,t} + x^{'}_{i,j,s,t} \geq 0 , \forall (i,j) \in E  , \forall s,t \in V'\\
& -x^{'}_{s,s,t} + x^{'}_{t,s,t} \geq 1  , \forall s,t \in V' \\

& \displaystyle\sum\limits_{j:(i,j) \in E'} y^{'}_{i,j,s,t} - \displaystyle\sum\limits_{k:(k,i) \in E'} y^{'}_{k,i,s,t} = \begin{cases} y^{'}_{s,t} &\mbox{if } i = s \\ -y^{'}_{s,t} &\mbox{if } i = t \\ 0 &\mbox{otherwise} \end{cases} \forall i,s,t \in V' \\
& y^{'}_{i,j,s,t} \geq -u_{i,j}, \forall (i,j) \in E' , \forall s,t \in V'\\
& \\
& x^{'}_{i,j,s,t} \geq 0 , \forall (i, j) \in E , \forall s,t \in V' \\
& x^{'}_{i,s,t} \textrm{ unrestricted}, \forall i, s, t \in V' \\
& y^{'}_{s,t} \leq 0 , \forall s,t \in V' \\
& y^{'}_{i,j,s,t} \leq 0 , \forall (i, j) \in E' , \forall s,t \in V'. \\
\end{array}
\label{MLIPDualSubproblem}
\end{equation}

At the beginning of each iteration $c$, the master is solved and we obtain the optimal values for $z_i$ and $w_{i,j}$. Initially, we start with an infinite objective function and all $z_i = 1$. The dual of the linear subproblem, shown in Equation~\ref{MLIPDualSubproblem}, is then solved with the optimal $w_{i,j}$'s substituted in. 

If the subproblem is unbounded, simplex returns the extreme ray, defining $\mathbf{x}^{'}_c$ and $\mathbf{y}^{'}_c$, and we add the constraint

\[\displaystyle\sum\limits_{\substack{s,t \in V' \\ s \neq t}} \displaystyle\sum\limits_{(i,j) \in E} u_{i,j}w_{i,j} x^{'}_{i,j,s,t,c}+\displaystyle\sum\limits_{\substack{s,t \in V'\\ s \neq t}} y^{'}_{s,t,c} -  \displaystyle\sum\limits_{\substack{s,t \in V' \\ s \neq t}} \displaystyle\sum\limits_{(i,j) \in E'} (1-w_{i,j}) y^{'}_{i,j,s,t,c} \geq 0.\]

If the subproblem has an objective function value less than or equal to the incumbent value of $\mathcal{L}_k$, then we add in the constraint

\[\displaystyle\sum\limits_{\substack{s,t \in V' \\ s \neq t}} \displaystyle\sum\limits_{(i,j) \in E}  u_{i,j}w_{i,j} x^{'}_{i,j,s,t,c}+\displaystyle\sum\limits_{\substack{s,t \in V'\\ s \neq t}} y^{'}_{s,t,c} -  \displaystyle\sum\limits_{\substack{s,t \in V' \\ s \neq t}} \displaystyle\sum\limits_{(i,j) \in E'}(1-w_{i,j}) y^{'}_{i,j,s,t,c} \geq \mathcal{L}_k.\]

Otherwise, the algorithm terminates.

Preliminary testing of the Benders decomposition of VIMAX reveals the same problem that plagues large instances of the MIP formulation: the objective function values of the linear subproblems encountered are quite large compared to the objective function value of any feasible integer solution.  Thus, the cuts added do not adequately constrain the master problem.  Future work is needed to develop improved Benders decompositions.

%
%


\bibliographystyle{plainnat}
\bibliography{SocialNetsv7.bib} 


\end{document}